\documentclass[11pt]{amsart}
\usepackage{geometry}        
\usepackage{tikz} % See geometry.pdf to learn the layout options. There are lots.
\tikzset{every picture/.style={line width=0.75pt}} %set default line width to 0.75pt        
\usepackage{graphicx}
\usepackage{amssymb}
\usepackage{epstopdf}
\usepackage{amsthm}
\newtheorem{thm}{Theorem}
\newtheorem{prop}{Proposition}
\newtheorem{lem}{Lemma}
\newtheorem{cor}{Corollary}

\title{Sharkovsky's Ordering in the Mandelbrot Set}
\author{Reila Zheng}

\begin{document}
\maketitle

\begin{abstract}
    Sharkovsky's ordering describes orbit forcing of interval maps, and generalizations of Sharkovsky's ordering exist for maps of trees. In this paper I will describe Sharkovsky's ordering and analogous orderings for trees, and their occurrence on the Mandelbrot set.
\end{abstract}

\section{Introduction}
Orbit forcing in dynamical systems is the phenomenon where the existence of orbits of a certain combinatorial type forces the existence of orbits of another combinatorial type. The original problem of orbit forcing on continuous maps of the interval was solved by Sharkovsky in the 1960s. %[cite]. 
In his celebrated theorem, he defined a total ordering $<_2$ on the natural numbers such that any interval map $f:I \rightarrow I$ with a periodic point of period $m$ also must have a periodic point of period $n$ for all $n <_2 m$. That is, the set of periods $Per(f)$ of periodic orbits of $f$ is a tail of this ordering. For a natural number $k \geq 2$, the $k$-Sharkovsky ordering is a partial ordering on $\mathbb{N}$ that describes the periods of periodic points of a continuous map on a tree $f: T \rightarrow T$. That is, if $T$ is a tree with $k$ leaves, then the set of periods $Per(f)$ of periodic orbits of $f$ is a union of tails of $i$-Sharkovsky orderings, for $2 \leq i \leq k$.

The $2$-Sharkovsky ordering manifests itself in the parameter space of real quadratic polynomials $z \mapsto z^2 + c$ where $c \in \mathbb{R}$. That is, for each $n \in \mathbb{N}$, let $c_n$ be the largest parameter such that the critical point of $c_n$ is of period $n$. Then $c_n < c_m$ iff $n >_2 m$. The goal of this paper is to generalize this ordering to the veins of the complex quadratic family.

A \emph{vein} on the Mandelbrot set is an embedded arc from the parameter $0$ in the main cardioid to a tip. 
We may assume that the vein intersects the boundary of every hyperbolic component in at most two points. 

Suppose $V$ is a vein of the Mandelbrot set $\mathcal{M}$. For hyperbolic components $C_1, C_2$ that intersect $V$, we say $C_1$ is \emph{closer} to the main cardioid than $C_2$ along $V$ if the connected component of $V\setminus C_2$ intersecting $C_1$ also intersects the main cardioid. For $C_1, C_2$ hyperbolic components on $V$, I will use the notation $C_1 \prec_V C_2$ if $C_1$ is closer to the main cardioid than $C_2$ along $V$. Since there are finitely many hyperbolic components of a given period, if there exist components of period $n$ which intersect $V$ nontrivially, there is a unique component of period $n$ closest to the main cardioid along $V$, which we will denote by $C_V(n)$. My results describe the ordering $\succ_V$ and dynamics of the parameters in the hyperbolic components $C_V(n)$ along $V$.

Every quadratic map $f_c(z) = z^2 + c$ has two fixed points, $\alpha$ and $\beta$, where $\beta$ is the landing point of the unique fixed external ray of angle 0. For natural numbers $k>p>0$ where $\gcd(p,k) = 1$, the $p/k$ limb of the Mandelbrot set is the set of parameters with combinatorial rotation number $p/k$ at $\alpha$ \cite{PR}.

The principal Misiurewicz point on this limb is the unique parameter $m_{p/k}$ that maps the critical point to the $\alpha$ fixed point with minimal preperiod $k-1$. Then $\mathcal{M} \setminus \{m_{p/k}\}$ is the union of $k$ connected components. Denote as $\mathcal{M}_{p/k,0}$ the connected component that contains the main cardioid. For each of the other connected components, there is a parameter $c_l = c_{p/k, l}$ that maps the critical point to the $\beta$ fixed point with minimal preperiod. Let $c_1$ be this parameter that maps to $\beta$ with minimal preperiod across all of $\mathcal{M} \setminus \mathcal{M}_{p/k,0}$. Then $c_1$ has preperiod $k-1$, and let $\mathcal{M}_{p/k,1}$ be the component that contains $c_1$. For each of the $k-2$ other connected components, consider the parameter that maps the critical point to the $\beta$ fixed point with minimal preperiod in that connected component. If the preperiod is $k+l-1$, then we will denote the parameter by $c_l$, with $l = 2, \dots, k-1$, and call the connected component that contains it $\mathcal{M}_{p/k, l}$. A \emph{$(k,l)$-vein} is a vein from the main cardioid to the $c_l$ parameter on the $p/k$ limb. 

When $l =1$, this is a principal vein. We will call the vein $V$ a secondary vein when $2 \leq l \leq k-1$. For example, in the $2/5$-limb, the principal vein is the vein from $0$ to $c_1$ the landing point of the ray with external angle of $5/16$ of preperiod $4$. There are three secondary veins, $V_l$ is the vein from $0$ to $c_l$ the landing point of the ray with external angle of $\theta_l$, where $\theta_2 = 19/64$ of preperiod 6, $\theta_3 = 41/128$ of preperiod 7, and $\theta_4 = 77/256$ of preperiod 8.

\begin{figure}[h!]
    \centering
    \includegraphics[width=0.5\linewidth]{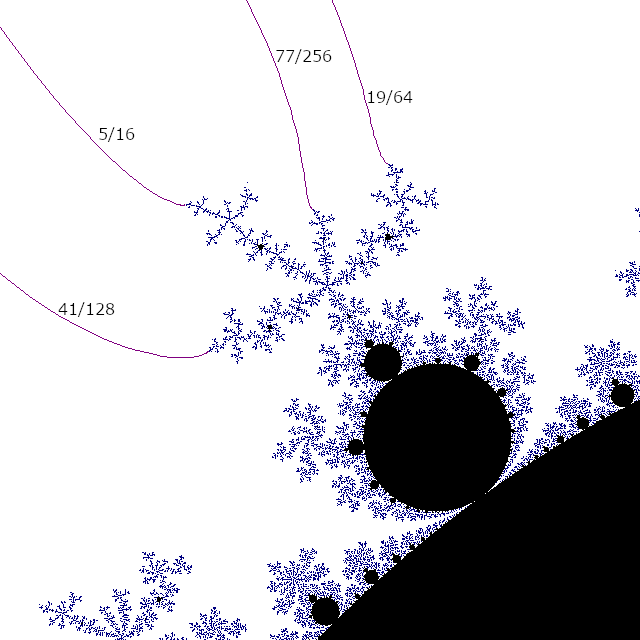}
    \caption{Principal and Secondary Veins in the 2/5-Limb.}
    \label{fig:enter-label}
\end{figure}

\begin{thm}[Ordering of Hyperbolic Components along a Vein]
\label{ordering}

Let $V$ be a $(k,l)$-vein and let $\mathbb{N}(k,l) := \{n \ : \ n \geq k+l\} \cup \{k,1\}$. Then:

\begin{enumerate}
\item $C_V(n)$ exists exactly for all $n \in \mathbb{N}(k,l)$. 

\item If $l =1$ and $n_1, n_2 \in \mathbb{N}(k,1)$ satisfy $n_1 >_k n_2$, then $C_V(n_1) \succ_V C_V(n_2)$. 

\item If $l \neq 1$ and $n_1, n_2 \in \mathbb{N}(k,l)$ satisfy $n_1 >_k n_2$ and $n_2 \leq_k k + l$, then $C_V(n_1) \succ_V C_V(n_2)$. 

\end{enumerate}
\end{thm}

\begin{thm} [Dynamics of Hyperbolic Components along a Vein]
\label{dynamics}
Let $V$ be a $(k,l)$-vein. Then:
\begin{enumerate}
    \item If $i \geq 1$ then the dynamics of $f_c$ for the centre $c$ of $C_V(ik+l)$ is a spiral graph and $Per(f_c) = \{m  \ : \ m \leq_k ik+l\}$. 

    \item For any $0 < l' < k$ with $l' \neq l$, the dynamics of $f_c$ for the centre $c$ of  $C_V(ik+l')$ is not a spiral graph.% and the ordering $\succ_V$ is coordinated by $\{n | n \equiv l \mod k\}$.

\end{enumerate}
\end{thm}

When it is clear from context, I will use the notation $C(n) = C_V(n)$. 

\begin{thm}[Explicit Description of the Vein Ordering]
\label{explicit}
Let $V$ be a $(k,l)$ vein.
We can obtain an explicit description of the vein ordering $\succ_V$ in some cases.

\begin{enumerate}
    \item If $k = 2l$, then the ordering of $C_V(n)$ for $n \leq_k k+l$ is explicitly determined as follows: 
\begin{center}
    \begin{align*}
    & C(k+l) \succ C(2k+l) \dots C(jk+l) \succ C((j+1)k+l) \dots\\
    & C(3k) \succ C(5k) \succ C(7k) \dots \succ C(4k) \succ C(2k) \succ C(k) \succ C(1)\\
    & \text{where the above line is given by the $k$ times the usual Sharkovsky ordering.}\\
    \end{align*}
\end{center}

\item If $k = 3l$ or $2k = 3l$, then the ordering of $C_V(n)$ for $n \leq_k k+l$ is explicitly determined as follows:
\begin{center}
    \begin{align*}
& C(k+l) \succ C(3k+2l) \succ C(4k+2l) \succ \\
& C(2k+l) \succ C(5k+2l) \succ C(6k+2l) \succ\\
& C(3k+l) \succ C(7k+2l) \succ C(8k+2l) \succ \\
& \dots \\
& C(jk+l) \succ C((2j+1)k + 2l) \succ C((2j+2)k + 2l) \succ \\
& C((j+1)k + l) \succ C((2(j+1)+1)k + 2l) \succ C((2(j+1)+2)k + 2l) \succ\\
& \dots \\
& C(3k) \succ C(5k) \succ C(7k) \dots \succ C(4k) \succ C(2k) \succ C(k) \succ C(1)\\
&\text{where the above line is given by $k$ times the usual Sharkovsky ordering.}\\
\end{align*}
\end{center}
\end{enumerate}

\end{thm}

\subsection{Background}

The problem of orbit forcing on trees has been described by Baldwin in the 1980s \cite{B}. In his setting, he considers a continuous map on a tree, and what can be said about the set of periods of periodic orbits of such a map. In his result, he describes the correspondence between maps on $k_0$-stars with a union of tails of $k$-Sharkovsky orderings, where $2 \leq k \leq k_0$.

In order to understand complex dynamics of quadratic maps using combinatorial descriptions of such dynamics on trees, Bruin and Schleicher have developed a correspondence between the dynamics on Hubbard trees and combinatorial objects such as internal addresses and kneading sequences \cite{BS}.

This paper is a novel application of Baldwin's orbit forcing results to the complex dynamics setting. We will not use the internal address approach in this result.

\section{Definitions and Notation}
\subsection{Complex Dynamics}
Consider the family of quadratic maps $f_c(z) = z^2 + c$, where $c\in \mathbb{C}$. The Mandelbrot set $\mathcal{M} \subseteq \mathbb{C}$ consists of all parameters $c$ such that the orbit of the critical point is bounded.

The external of the Mandelbrot set $\mathbb{C} \setminus \mathcal{M}$ can be uniformized to the external of the closed disk $\mathbb{C} \setminus \mathbb{\overline D}$ so that $\frac{\Phi(c)}{c}=1$ as $c \rightarrow \infty$ by the unique conformal isomorphism $\Phi: \mathbb{C} \setminus \mathcal{M} \rightarrow \mathbb{C} \setminus \mathbb{\overline D}$. Then the pull back of rays on the external of the disk are the external rays of the Mandelbrot set, or parameter rays, and we denote as $R^\theta =  \Phi^{-1}(\{re^{i\theta} | r >1\})$.

We say that a parameter ray $R^\theta$ lands at a point $c$ on the boundary of the Mandelbrot set if $\lim_{r\rightarrow 1^+} \Phi^{-1}(re^{i\theta}) = c$. Let $R^{\theta+}$ and $R^{\theta-}$ be external rays on the Mandelbrot set such that they land at a common point $c \in \partial \mathcal{M}$. Then $R^{\theta+}\cup R^{\theta-}\cup {c}$ divides $\mathbb{C}$ into two components.

Given a parameter $c \in \mathbb{C}$ we can define the filled Julia set $K_c$ as the set of all points $z$ in the dynamical plane such that the orbit of $z$ under $f_c$ remains bounded. For parameters $c \in \mathcal{M}$, $K_c$ is connected and simply connected, so its external $\mathbb{\widehat C} \setminus K_c$ can be uniformized to the external of the closed disk $\mathbb{\widehat C} \setminus \mathbb{\overline D}$, and there is a unique conformal isomorphism $\Psi_c: \mathbb{\widehat C} \setminus K_c \rightarrow \mathbb{\widehat C} \setminus \mathbb{\overline D}$ such that the dynamics of $f_c$ on the $\mathbb{C} \setminus K_c$ commutes with the $z \mapsto z^2$ on the external of the disk. Then the pull back of rays on the external of the disk are the external rays of the filled Julia set, or dynamical rays, and we denote as $R_c^\theta =  \Psi^{-1}(\{re^{i\theta} | r >1\})$.

For a parameter $c$, the Julia set $J_c = \partial K_c$ is the boundary of the filled Julia set. We say that a dynamical ray $R_c^\theta$ lands at a point $z$ on the Julia set $J_c$ if $\lim_{r\rightarrow 1^+} \Psi_c^{-1}(re^{i\theta}) = z$.

\subsubsection{Orbit Portrait}
An orbit portrait of a periodic orbit $\{z_1, \dots, z_n\}$ for $f_c$, is a set of sets of angles $A_1, \dots, A_n$, where $A_i$ are all the angles $\theta_1, \dots, \theta_k$ such that the dynamical rays $R_c^{\theta_1}, \dots R_c^{\theta_k}$ land at $z$. The valence of a point $z_i$ in the orbit is the number of dynamical rays landing at $z_i$, and since only points with valence $\geq 2$ can be on the Hubbard tree $H_c$, we are only interested in orbit portraits with valence $\geq 2$.

A hyperbolic component $C$ of the Mandelbrot set is a connected component of the interior of the Mandelbrot set, such that every parameter in $C$ has an attracting periodic cycle.

Suppose that the parameter rays $R^{\theta+}$ and $R^{\theta-}$ land at a root $r \neq \frac{1}{4}$ of a hyperbolic component $C$. Moreover, suppose that $R^{\theta+}\cup R^{\theta-}\cup {r}$ divides $\mathbb{C}$ into two connected open components, $W$, and $\mathbb{C} \setminus \overline W$ such that $C \supseteq W$ and $\mathbb{C} \setminus \widehat W \supseteq C(1)$. Then we say $W = W_C$ is the wake of the hyperbolic component $C$.

A parameter $c \in \mathbb{C}$ has a periodic orbit with orbit portrait $P$ if and only if $c \in W\cup\{r\}$. If $P$ and $Q$ are distinct orbit portraits, then the wakes $W_P$ and $W_Q$ are either disjoint or strictly nested (ie, either $W_P \cap W_Q = \emptyset, \overline W_P \subseteq W_Q$, or $\overline W_Q \subseteq W_P$). If $c_1$ and $c_2$ are parameters for which $V$ is a vein between them, then the wakes of roots of hyperbolic components along $V$ are all nested.

Suppose $f_c$ has an orbit with portrait $P$. Then there is hyperbolic component $C$ with root $r_P$ such that $f_{r_P}$ has a parabolic orbit with portrait $P$, and there are parameter rays $R^{\theta+}, R^{\theta-}$ that land at $r_P$. Moreover these rays divide $\mathbb{C}$ into $W_P$ and $\mathbb{C}\setminus W_P$, where $W_P \supseteq C$ is the wake of the orbit portrait $P$.

A hyperbolic component $C$ is visible from $C'$ if $V = [C, C']$ is a vein, $C \succ_V C'$, and all components $C'' \in (C, C')$ have the property that $per(C'') > per(C)$. At each hyperbolic component $C$ with internal angle $p/q$, there is a tree $T_{p/q}$ consisting of all hyperbolic components visible from $C$ in the $p/q$ sub-wake of $C$. There are finitely many such hyperbolic components on each $T_{p/q}$, 

\subsubsection{Quasiconformal Surgery of Sectors}
A vein is a continuous embedding of an arc inside $\mathcal{M}$. Denote the real vein $V_R = \mathcal{M} \cap \mathbb{R}$.

Let $k \geq 3$ and let $1 \leq p \leq k-1$ be relatively prime to $k$.\\

Let $c \in V_R$ be a parameter, and consider the dynamical plane of $f_c$. We will cut the dynamical plane along the rays at the $\alpha$ fixed point. The sector that contains the critical point will be denoted as the critical sector, $S_{c,R}$ and the other sector will be the non-critical sector $S_{n,R}$. 

Consider the plane divided into $k$ sectors labelled $S_0, S_1, \dots S_{k-1}$ at angles $0, p/k, 2p/k, \dots, (k-1)p/k$ respectively. Embed the critical sector in $S_0$ via the map $\phi_0: S_{c,R} \rightarrow S_0$.

For each $1 \leq i \leq k-1$, embed the non-critical sector in $S_i$ via the map $\phi_i: S_{n,R} \rightarrow S_i$.

We will duplicate the dynamics of the non-critical sector $k-2$ times, at the sectors $S_{1}, S_{2}, \dots, S_{k-1}$.

The new map under the surgery $f_{\Phi_{p,k}(c)}$ will be defined by

$$f_{\Phi_{p,k}(c)} = \phi_{i+1} \circ f \circ \phi_i^{-1}$$

on the $i$-th sector for $i = 1, \dots, k-1$ and we take $\phi_{k} = \phi_0$. For $i=0$, define $f_{\Phi_{p,k}(c)}$ to be $\phi_{0} \circ f \circ \phi_0^{-1}|_{f^{-1}(S_{c,R})}$ and $\phi_{1} \circ f \circ \phi_0^{-1}|_{f^{-1}(S_{n,R})}$.

That is, $f_{\Phi_{p,k}(c)}$ duplicates the dynamics of the non critical sector $k-1$ times and maps $S_1 \rightarrow S_2 \rightarrow \dots \rightarrow S_{k-1} \rightarrow S_{0}$.%finish this

Define $\Phi_{p,k}$ to me the map from $V_R$ to its image, which is a principal $p/k$-vein. This is an embedding of $V_R$.

\subsection{Sharkovsky Ordering}
\subsubsection{Interval Maps}
An interval map is a continuous function $f:I\rightarrow I$ on a closed interval. The set of periods of $f$ is $Per(f) : = \{n | f \text{ has a periodic point of exact period } n \}$.

The \emph{Sharkovsky ordering} $>_2$ is an ordering on $\mathbb{N}$ that describes orbit forcing and realization of interval maps. It is given by $3 >_2 5 >_2 7 >_2 9 >_2 11 >_2 \dots 2\times 3 >_2 2\times 5 >_2 \dots 2^n\times 3 >_2 2^n\times 5 >_2 \dots >_2 2^n >_2 2^{n-1} >_2 \dots >_2 2 >_2 1$.

A tail of the $>_2$ ordering is a set $S \subseteq \mathbb{N}$ such that whenever $s \in S$ and $t \in \mathbb{N} \setminus S$ then $t >_2 s$. Every interval map $f$ has its set of periods $Per(f)$ given by a tail of $>_2$, and every tail of this ordering is realized as the set of periods of some interval map.

Suppose $n$ is odd and $p_0 \mapsto p_1 \mapsto \dots \mapsto p_n = p_0$ is a period $n$ cycle of $f: I \rightarrow I$. Then $\{p_1, \dots, p_n\}$ is a \emph{\v{S}tefan cycle} if the points in the cycle occur in the following configuration along $I$ or in the reverse orientation: $p_{n-1} < p_{n-3} < \dots < p_4 < p_2 < p_0 = p_n < p_1 < p_3 < \dots < p_{n-4} < p_{n-2}$. If $f$ has a period $n$ orbit, then it has a period $n$ orbit that is a \v{S}tefan cycle.

\subsubsection{Star Maps}
For $k \geq 3$, a $k$-star $T$ is a topological tree with $k$ leaves joined to one vertex of degree $k$.

The $k$-Sharkovsky ordering is a partial ordering on $\mathbb{N}$ that describes the orbit forcing and realization of continuous maps on stars. The $k$-Sharkovsky ordering is a partial ordering defined by:

$\forall n, m in S_k$
\begin{enumerate}
    \item if $n > 1$ then $n >_k 1$,
    \item if $n, m \equiv 0 \mod k$ and $n/k >_2 m/k$, then $n >_k m$,
    \item if $n \not \equiv 0 \mod k$ and $m = in + jk$ for some $i\geq 0, j \geq 1$ then $n >_k m$.
\end{enumerate}

A tail of the $>_k$ ordering is a set $S \subseteq \mathbb{N}$ such that whenever $s \in S$ and $t \in \mathbb{N} \setminus S$ then $t \not \leq_k s$. Every continuous map on a $k$-star has its set of periods given by a union of tails of $>_i$ for $i = 2, \dots, k$, and every such union of tails is realized as the set of periods of a continuous map on a $k$-star.

Suppose $T$ is a $k$-star with a continuous map $f: T \rightarrow T$ that fixes the central degree $k$ vertex $\alpha$ of $T$, such that $p_0 \mapsto p_1 \mapsto \dots \mapsto p_n=p_0$ is a period $n$ cycle of $f$, where $n = ik + j$ for some $i \geq 1$ and $1 \leq j \leq n-1$.
Then $\{p_1, \dots, p_n\}$ is a \emph{spiral cycle} if the points in the cycle occur in an outward spiral on all $k$ arms of $T$ as evenly as possible in the following way: Denote the leaves of $T$ as $L_1, \dots, L_k$, with $L_i = [\alpha, v_i]$. Then the vertices $p_j \in L_i$ if $j \equiv i \mod k$. Moreover if $j \leq n-k$ then $p_j \in [\alpha, p_{j+k})$. %A spiral cycle is a periodic orbit $\{p_1, \dots, p_n = p_0\}$ where $f(p_i) = p_{i+1}$ for all $i$.

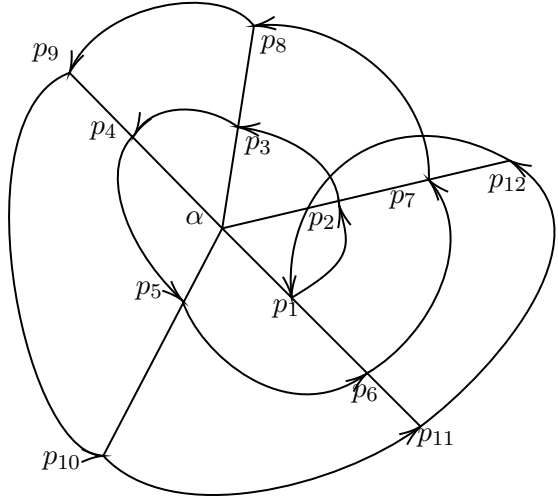
\begin{figure}
\centering

\begin{tikzpicture}[x=0.75pt,y=0.75pt,yscale=-1,xscale=1]

%Straight Lines [id:da9866100188716722] 
\draw    (100,115) -- (200,215) ;
%Straight Lines [id:da803959939475771] 
\draw    (100,115) -- (245,80.81) ;
%Straight Lines [id:da4769486242697454] 
\draw    (100,115) -- (40,229.81) ;
%Straight Lines [id:da24908911497183528] 
\draw    (100,115) -- (116,12.81) ;
%Straight Lines [id:da8002136009456067] 
\draw    (23,36.5) -- (100,115) ;
%Curve Lines [id:da8503594543619939] 
\draw    (135,150) .. controls (167.01,131.08) and (164.21,123.93) .. (159.45,105.27) ;
\draw [shift={(159,103.5)}, rotate = 75.96] [color={rgb, 255:red, 0; green, 0; blue, 0 }  ][line width=0.75]    (10.93,-3.29) .. controls (6.95,-1.4) and (3.31,-0.3) .. (0,0) .. controls (3.31,0.3) and (6.95,1.4) .. (10.93,3.29)   ;
%Curve Lines [id:da14875012945215427] 
\draw    (173,188.2) .. controls (205.51,168.99) and (229.47,121.07) .. (205.34,91.92) ;
\draw [shift={(204.2,90.6)}, rotate = 48.12] [color={rgb, 255:red, 0; green, 0; blue, 0 }  ][line width=0.75]    (10.93,-3.29) .. controls (6.95,-1.4) and (3.31,-0.3) .. (0,0) .. controls (3.31,0.3) and (6.95,1.4) .. (10.93,3.29)   ;
%Curve Lines [id:da8652722274248088] 
\draw    (200,215) .. controls (239.6,185.3) and (299.78,114.93) .. (246.65,81.8) ;
\draw [shift={(245,80.81)}, rotate = 30.27] [color={rgb, 255:red, 0; green, 0; blue, 0 }  ][line width=0.75]    (10.93,-3.29) .. controls (6.95,-1.4) and (3.31,-0.3) .. (0,0) .. controls (3.31,0.3) and (6.95,1.4) .. (10.93,3.29)   ;
%Curve Lines [id:da07924870626116509] 
\draw    (159,103.5) .. controls (158.03,81.07) and (132.33,67.21) .. (109.73,64.12) ;
\draw [shift={(108,63.91)}, rotate = 6.43] [color={rgb, 255:red, 0; green, 0; blue, 0 }  ][line width=0.75]    (10.93,-3.29) .. controls (6.95,-1.4) and (3.31,-0.3) .. (0,0) .. controls (3.31,0.3) and (6.95,1.4) .. (10.93,3.29)   ;
%Curve Lines [id:da8791318416882075] 
\draw    (204.2,90.6) .. controls (206.17,44.3) and (158.46,8.59) .. (117.85,12.6) ;
\draw [shift={(116,12.81)}, rotate = 352.62] [color={rgb, 255:red, 0; green, 0; blue, 0 }  ][line width=0.75]    (10.93,-3.29) .. controls (6.95,-1.4) and (3.31,-0.3) .. (0,0) .. controls (3.31,0.3) and (6.95,1.4) .. (10.93,3.29)   ;
%Curve Lines [id:da6350738984658868] 
\draw    (108,63.91) .. controls (91.59,51.93) and (66.81,51.52) .. (57.01,65.88) ;
\draw [shift={(56,67.5)}, rotate = 299.36] [color={rgb, 255:red, 0; green, 0; blue, 0 }  ][line width=0.75]    (10.93,-3.29) .. controls (6.95,-1.4) and (3.31,-0.3) .. (0,0) .. controls (3.31,0.3) and (6.95,1.4) .. (10.93,3.29)   ;
%Curve Lines [id:da3728735179211682] 
\draw    (116,12.81) .. controls (99.26,-9.35) and (38.85,1.17) .. (23.66,34.94) ;
\draw [shift={(23,36.5)}, rotate = 291.8] [color={rgb, 255:red, 0; green, 0; blue, 0 }  ][line width=0.75]    (10.93,-3.29) .. controls (6.95,-1.4) and (3.31,-0.3) .. (0,0) .. controls (3.31,0.3) and (6.95,1.4) .. (10.93,3.29)   ;
%Curve Lines [id:da49624345436430284] 
\draw    (56,67.5) .. controls (38.27,83.26) and (48.68,121.34) .. (78.62,150.19) ;
\draw [shift={(80,151.5)}, rotate = 223.09] [color={rgb, 255:red, 0; green, 0; blue, 0 }  ][line width=0.75]    (10.93,-3.29) .. controls (6.95,-1.4) and (3.31,-0.3) .. (0,0) .. controls (3.31,0.3) and (6.95,1.4) .. (10.93,3.29)   ;
%Curve Lines [id:da9169737219511354] 
\draw    (23,36.5) .. controls (-33.15,57.19) and (-1.97,224.37) .. (38.16,229.68) ;
\draw [shift={(40,229.81)}, rotate = 180.44] [color={rgb, 255:red, 0; green, 0; blue, 0 }  ][line width=0.75]    (10.93,-3.29) .. controls (6.95,-1.4) and (3.31,-0.3) .. (0,0) .. controls (3.31,0.3) and (6.95,1.4) .. (10.93,3.29)   ;
%Curve Lines [id:da5591525806075952] 
\draw    (80,151.5) .. controls (87.92,177.24) and (132.1,217.39) .. (171.8,189.08) ;
\draw [shift={(173,188.2)}, rotate = 143.13] [color={rgb, 255:red, 0; green, 0; blue, 0 }  ][line width=0.75]    (10.93,-3.29) .. controls (6.95,-1.4) and (3.31,-0.3) .. (0,0) .. controls (3.31,0.3) and (6.95,1.4) .. (10.93,3.29)   ;
%Curve Lines [id:da72049181720931] 
\draw    (40,229.81) .. controls (74.65,267.12) and (158.31,245.44) .. (198.79,215.9) ;
\draw [shift={(200,215)}, rotate = 143.13] [color={rgb, 255:red, 0; green, 0; blue, 0 }  ][line width=0.75]    (10.93,-3.29) .. controls (6.95,-1.4) and (3.31,-0.3) .. (0,0) .. controls (3.31,0.3) and (6.95,1.4) .. (10.93,3.29)   ;
%Curve Lines [id:da25193422796939147] 
\draw    (245,80.81) .. controls (181.64,44.86) and (130.04,92.53) .. (134.84,148.31) ;
\draw [shift={(135,150)}, rotate = 263.94] [color={rgb, 255:red, 0; green, 0; blue, 0 }  ][line width=0.75]    (10.93,-3.29) .. controls (6.95,-1.4) and (3.31,-0.3) .. (0,0) .. controls (3.31,0.3) and (6.95,1.4) .. (10.93,3.29)   ;

% Text Node
\draw (124,149.4) node [anchor=north west][inner sep=0.75pt]    {$p_{1}$};
% Text Node
\draw (142,105.4) node [anchor=north west][inner sep=0.75pt]    {$p_{2}$};
% Text Node
\draw (110,67.31) node [anchor=north west][inner sep=0.75pt]    {$p_{3}$};
% Text Node
\draw (32,59.4) node [anchor=north west][inner sep=0.75pt]    {$p_{4}$};
% Text Node
\draw (55,140.4) node [anchor=north west][inner sep=0.75pt]    {$p_{5}$};
% Text Node
\draw (164,192.4) node [anchor=north west][inner sep=0.75pt]    {$p_{6}$};
% Text Node
\draw (183,94.4) node [anchor=north west][inner sep=0.75pt]    {$p_{7}$};
% Text Node
\draw (118,16.21) node [anchor=north west][inner sep=0.75pt]    {$p_{8}$};
% Text Node
\draw (3,20) node [anchor=north west][inner sep=0.75pt]    {$p_{9}$};
% Text Node
\draw (8,226.4) node [anchor=north west][inner sep=0.75pt]    {$p_{10}$};
% Text Node
\draw (197,214.4) node [anchor=north west][inner sep=0.75pt]    {$p_{11}$};
% Text Node
\draw (233,86.4) node [anchor=north west][inner sep=0.75pt]    {$p_{12}$};
\draw (80,105) node [anchor=north west][inner sep=0.75pt]    {$\alpha $};
\end{tikzpicture}

\caption{Schematic of a $5$-star with a spiral cycle of period 12.}
\end{figure}

Suppose $T$ is tree with vertices$\{p_1, \dots, p_n\} \cup \{\alpha\}$. Then $f: T \rightarrow T$ is a \emph{spiral graph} if $f$ maps $\{p_1, \dots, p_n\}$ in a spiral cycle, $f$ fixes $\alpha$, and $f$ maps every edge $[v,w]$ of $T$ homeomorphically to the path $[f(v), f(w)]$.

\subsubsection{Markov Graphs}

Suppose $T=(V,E)$ is a $k$-star with a fixed central vertex $\alpha$ and $f:T \rightarrow T$ maps vertices to vertices and maps edges to unions of edges homeomorphically. We can define the Markov graph of $f$ (or equivalently of $T$) as the directed graph $G_T = (E,A)$, where there is an arc $A_{i,j} = (E_i, E_j)$ from $E_i$ to $E_j$ if $f(E_i) \supseteq E_j$.

A nonrepetitive loop of length $m$ is a loop in the Markov graph $L = E_{1,0}, E_{i_1}, E_{i,2}, \dots, E_{i,m} = E_{i,0}$ such that is not the union of $l$ repeats of a smaller loop.

Suppose moreover that $f: T\rightarrow T$ maps $p_i$ to $p_{i+1}$ for all $i$, and fixes $\alpha$ is a spiral graph with $k$ leaves and $p$ vertices, where $p \equiv i \mod k$ for some $0 < i \leq p-1$. Denote the spiral cycle that defines $f$ to be $p_1, p_2, p_3, \dots, p_n = p_0$ where $f(p_i) = p_{i+1}$ with $p_1, \dots p_k$ the vertices closest to $\alpha$ on each branch.

Define the edge $I_j$ to be the unique edge through vertex $p_j$ that is closest to $\alpha$. Then the Markov graph of $f$ consists of the following $p$-cycle: $I_0 \rightarrow I_1 \rightarrow I_2 \rightarrow \dots \rightarrow I_{p-1} \rightarrow I_{p}=I_{0}$ and arcs $I_p \rightarrow I_i, I_p \rightarrow I_{k+i}, \dots, I_p \rightarrow I_{p-k}$, $I_p \rightarrow I_1$, and $I_k \rightarrow I_1$.

Suppose $k \neq m \neq 1$. Then $f$ has a point of period $m$ if $G_T$ has a nonrepetitive loop of length $m$ that does not consist purely of edges adjacent to $\alpha$.  $G_T$ has a loop of length $m$ if $f$ has a point of period $m$.

Suppose $T$ is a $k$-star with vertices $\alpha$ of degree $k$ and $p_0, \dots, p_{n-1}$ of degree at most 2. Moreover suppose $f: T \rightarrow T$ has a spiral cycle $p_0 \mapsto p_1 \mapsto \dots \mapsto p_n = p_0$ of period $n > 1$, $n \not \equiv 0 \mod k$, and fixed point $\alpha$ of degree $k$, and $f$ maps any edge between these vertices homeomorphically. Then $Per(f) = \{m | m \leq_k n\}$.

\section{Results}

First, we will show that the ordering of $C_V(n)$ along the real vein is the $2$-Sharkovsky ordering. In this section, when referring to the dynamics of $f_c$ where $c$ is the centre of a hyperbolic component $C$, I will sometimes use $C$ instead of $f_c$, such as in the context of $Per(C)$.

\begin{prop}[The Real Vein]
\label{real}
Let $V = V_R = [-2,0]$. 

\begin{enumerate}

\item $C_V(n)$ exists for all $n \in \mathbb{N}$.

\item $C_V(n) \succ_{V} C_V(m)$ iff $n >_2 m$. 

\item The set of periods of any $T_{C_V(n)}$ is $\{m | m \leq_2 n\}$. 

\item If $n$ is odd, then $C_V(n)$ is a \v{S}tefan cycle.

\end{enumerate}
\end{prop}

\begin{proof}

\begin{enumerate}
\item If $c \in V$ then the Hubbard tree of $c$ is a line segment. Since the airplane $c$ is a parameter of period 3 in $V_R$, $T_c$ contains periodic orbits of all periods. Since periodic orbits in $T_c$ correspond to orbit portraits of distinct wakes nested along $V$. For every $n$ there is a unique parameter of period $n$ that is the root of the wake of the orbit portrait of period $n$ closest to the main cardioid along $V$. So $C_V(n)$ exists for every $n$.

\item By definition, $T_{C_V(n)}$ contains a periodic point of period $n$, so it contains a periodic point of period $m$ for all $m \leq_2 n$. Corresponding to each of these periodic orbits of period $m$ is its wake. Note that for the wake $W_{C_V(n)}$ of $C_V(n)$ is contained in the wake $W_{C_V(m)}$ of $C_V(m)$ iff $C_V(n) \succ_V C_V(m)$. Since these wakes are nested along $V$, let $W_{C}$ be the wake of period $m$ along $V$ that contains all others, where $C$ is the hyperbolic component corresponding to the root of $W_C$. Note that there is no other wake containing $W_{C}$ that corresponds to a period $m$ orbit, so $C = C_V(m)$. Since $m <_2 n$ implies $C_V(m) \prec_V C_V(n)$, and $<_2$ is a total ordering, the implication goes both ways.

\item Every periodic orbit of $T_{C_V(n)}$ corresponds to a wake of a hyperbolic component of that period that contains $C_V(n)$. For any $m >_2 n$, since $C_V(m) \succ_V C_V(n)$, all of the wakes of period $m$ orbits are contained in $W_{C_V(n)}$, thus there are no wakes of period $m >_2 n$ that contain $C_V(n)$. So the set of periods of $T_{C_V(n)}$ is $\{m | m \leq_2 n\}$.

\item By a theorem of Burns and Hasselblatt \cite{BH}, if $f$ is an interval map with $Per(f) = \{m | m \leq_2 2i+1\}$ where $i \geq 0$ then $f$ every orbit of period $2i+1$ is a \v{S}tefan cycle.
Consider $T_{C_V(2i+1)}$, since $Per(T_{C_V(2i+1)}) = \{m | m \leq_2 2i+1\}$ its critical orbit is a \v{S}tefan cycle of this period.

\end{enumerate}
\end{proof}

Next, we will use a lemma to derive the statement for principal veins of $p/k$-limbs. We will use quasiconformal surgery of sectors to obtain an embedding of $V_R$ into any $p/k$-principal vein.

\begin{lem} [From Real to Principal Vein]
\label{principal}
Let $V$ be a $(k,1)$-vein on the $p/k$-limb. Then $T(c)$ is a $k$-star for all $c \in V\setminus C_V(1)$. Moreover $C_V(ik+1)$ are spiral graphs for all $i \geq 0$.

\end{lem}

\begin{proof}
Consider the tip $c = -2$ of $V_R$. Its critical orbit is preperiodic to the $\beta$ fixed point with preperiod $1$, that is $\{0, c = -2, \beta = 2\}$. Then the $\alpha$ fixed point exists withing the edge $(-2, 0)$ of $T_c$. Let $\Phi_{p/k}: V_R \rightarrow V_{p/k, 1}$ be the quasiconformal surgery of sectors that maps $V_R$ onto the principal vein  $V_{p/k,1}$ in the $p/k$-limb. Then $c_{p/k, 1} = \Phi_{p/k}(c)$ is a parameter at the tip of $V_{p/k, 1}$ such that the critical orbit of $c_{p/k,1}$ maps to its beta fixed point with preperiod $k-1$. This is the Misiurewicz point in the $p/k$-limb that maps to $\beta$ with minimal preperiod \cite{T}.

Note that $T_{c_{p/k,1}}$ is a $k$-star, with critical orbit $\{c_0 = 0, c_1, \dots , c_k = \beta\}$ such that $c_i \mapsto c_{i+1}$ for all $0 \leq i \leq k-1$ and $c_k \mapsto c_k$, so it is the unique parameter of minimal preperiod $k-1$ that maps to $\beta$.

For any $i \geq 1$, and any $c \in C_{V_R}(2i+1)$, $T_c$ has a \v{S}tefan cycle as its critical orbit, say $\{c_0 = 0, c_1, c_2, \dots, c_{2i+1}\}$ with $c_j \mapsto c_{j+1}$ for all $j \geq 0$. Then their relative position along the interval $T_{C_{V_R}(2i+1)} = [c, f_c(c)]$ is given by $c_1 < c_{2i} < \dots < c_6 < c_4 < c_3 < c_5 < \dots <c_{2i-1} < c_0 = c_{2i+1} < c_2$. Then the $\alpha$ fixed point occurs between $c_3$ and $c_4$ (or in the case $i =1$, between $c_0$ and $c_1$). If we denote $I_0 = [\alpha, c_2]$ the critical sector of the interval, and $I_1 = [c_1, \alpha]$ the non-critical sector, then there are $i$ points of the critical orbit in $I_1$ and $i+1$ in $I_0$.

The dynamics of $\Phi_{p/k}(c)$ can be obtained by duplicating the non-critical sector ${k-2}$ with rotation number $p/k$ about $\alpha$.
Denote by $I_2, \dots, I_{k-1}$ copies of $I_1$, attached at $\alpha$, and embedded into the plane at a combinatorial rotation number of $p/k$ about $\alpha$. For each $1 \leq j \leq k-2$ let $\phi_{j, j+1}: I_j \rightarrow I_{j+1}$ be a homemorphjsm onto jts jmage, wjth $\phi_{j, j+1}(\alpha) = \alpha$. Then he dynamics of $f_{\phi{p/k}(c)}$ is given by $\phi_{j, j+1}: I_j \rightarrow I_{j+1}$ for $0 \leq j \leq k-2$ and $f_c|_{I_1}\circ \phi_{1, 2}^{-1} \circ \phi_{2,3}^-1 \circ \dots \circ \phi_{k-2,k-1}^{-1}: I_{k-1} \rightarrow I_0$, $f_c|_{T_c}: I_0 \rightarrow I_0 \cup I_{1}$. Thus $T_{\Phi_{p/k}(c)}$ is a spiral graph with a critical orbit of period $jk+1$.

\end{proof}

\begin{lem} [From Principal to Secondary Vein]
\label{secondary}
For each $1 \leq l \leq k-2$ let $V_l$ be a $(k,l)$ vein on the $p/k$-limb. Then there is a homeomorphism $\Psi_l=\Psi_{p/k,l}$ of $V_l$ to the $(k,l+1)$-vein $V_{l+1}$ of this limb, such that for all $i \geq 1$, $\Psi_l(C_{V_l}(ik+l)) = C_{V_{l+1}}(ik+l+1)$ and this is a spiral graph.
    
\end{lem}

\begin{proof}
%For ease of notation we will assume that $p = 1$.\\
We will use the homeomorphism $\Psi_l$ from $V_l$ to $V_{l+1}$ as described by Riedl in his thesis. If we remove $m_{p/k}$ the principal Misiurewicz point from $\mathcal{M}$, and remove the connected component of the main cardioid, we obtain $k-1$ connected components. In each of these connected components there is a Misiurewicz point that maps to $\beta$ with minimal preperiod. Then $\Psi_l$ maps the vein $V_l$ of the the $l$-th connected component to the vein $V_{l+1}$ of the $l+1$-th connected component and sends $c_l$ to $c_{l+1}$, the Misiurewicz points of their respective sector of minimal preperiods prefixed to $\beta$.

I will briefly describe $\Psi_l$ as a transformation of $c \in C_V(ik+l)$ to a parameter $\Psi_l(c)$ hybrid equivalent to $C_{V_{i+1}}(ik+l+1)$. Let $T_c$ be the Hubbard tree of $f_c$ embedded in the plane, and suppose it is a $k$-star, with the critical orbit $c_0 = 0, c_1 = c, c_2, \dots, c_{ik+l-1}$ and suppose it is a spiral cycle. Denote the edges of $T_c$ as $I_1, I_2, \dots, I_ik+l-1$, where $I_j$ is the edge closest to $\alpha$ incident to the vertex $c_j$. Consider the Markov graph, since $I_0 \rightarrow I_1 \rightarrow I_2 \rightarrow \dots \rightarrow I_{k-1}$, and $I_{k-1}$ maps to the union of edges which make up the segment $[c_1, c_k]$. Since $c_1$ and $c_k$ are on different sectors of $T_c$, $\alpha \in [c_1, c_k]$, thus there is a point $\alpha_{-k} \in I_1$ such that $f^{k}(\alpha_{-k}) = \alpha$, where $\alpha_{-k}$ is a preperiod $k-1$ point that maps to $\alpha$. %In fact, this point has portrait given by the principal Misiurewicz point $m_{p/k}$. 
Let $\alpha_{-k}, \alpha_{-k+1}, \alpha_{-k+2}, \dots \alpha_-1$ be the preperiodic points of $\alpha_{-k}$.

Consider $T_c$ embedded in the plane as a subset of $K_c$. Since $I_{2k}$ maps to the union of edges $I_{2k+1}, I_{k+1}$ (or if $i = 1$, reduce all indices mod $k+1$ so we have $I_{k-1}$ maps to $I_{k}, I_0$), there is a point $c_{k+1}' \in I_{2k}$ that maps to $c_{k+1}$. Let $I' = [\alpha, c_{k+1}']$, and let $J_{k-1+j} = f^{-j}(I')$ where we take the $j$-th preimage from $\alpha_{-j}$ for $1 \leq j \leq k$. Define $\overline{T_c} = T_c \bigcup \cup_{1 \leq j \leq k J_j}$, then each $\alpha_{-j} \in \overline{T_c}$ is of degree $3$. Let $c' = f_c^{-k}(c_{k+1}')$, the newly created leaf in $J_1$.

Define $T_c' = \overline{T_c}\setminus{\cup_{1 \leq j \leq k}(\alpha_{-j},c_j]}$. Define $f_{c'} : T_{c'} \rightarrow T_{c'}$  by $c_0 \mapsto c'$, $f_{c'}|{J_j} = f_c$ for $1 \leq j \leq k$. For all other vertices in $T_c \cap T_{c'}$, $f_c = f_{c'}$, and map the edges homeomorphically. This results in the following period $ik+l+1$ critical orbit: $\{c' = c'_1, f_{c}(c') = c'_2, \dots, f_c^{k+1}(c') = c_{k+1} = c'_{k+2}, \dots, f_c^{k+l-1}(c') = c_{k+l-1} =c'_{k+l}\}$. Thus this $T_{c'}$ is a spiral graph.

\end{proof}

\begin{prop} 
\label{visible}
Let $V$ be a $(k,l)$-vein. Then $C_V(k+l)$ is a narrow component in the tree of visible hyperbolic components in the $1/2$-wake of $C_V(k)$.
\end{prop}

\begin{proof}
 All hyperbolic components adjacent to the main cardioid are narrow. Since $C_V(k)$ is narrow, its tree of visible components in its $1/2$-wake consist of exactly $k$ hyperbolic components of period $k+1, k+2, \dots, 2k$, by Lau-Schleicher. Moreover the tree of visible components from $C_V(k)$ is a tree with $k-1$ branches, with $k-1$ leaves, and the component of period $2k$ occurs at internal angle $1/2$ from $C_V(k)$.
 %By the previous parts, for any $(k,l)$-vein $V$, $Per(C_V(k+l) = \{m| m \leq_k k+l\} = \{ik+jl | i \geq 1, j\geq 0\}\cup \{1\}$. Since the only component on $V$ of period $k$ is $C_V(k)$, the set of periods of hyperbolic components strictly between $C_V(k+l)$ and $C_V(k)$, that is $C \in (C_V(k+l), C_V(k))$, must all be greater than $k+l$. Hence $C_V(k+l)$ must be the visible component of period $k+l$ in this tree of visible components. 

Let us denote $V_l$ by the $(k,l)$-vein in the $p/k$-limb. Putting together all $k-1$ visible components of period $k+1, \dots, 2k-1$ along each of $V_1, \dots, V_{k-1}$, along with the component $C_{V_1}(2k)$ we obtain the tree of visible components. This tree consist of $k-1$ leaves, one for each of the $C_{V_l}(k+l)$, with the root component $C_{V_1}(2k)$.

Suppose it is not the case that $C_{V_l}(k+l)$ is narrow. Then $C_{V_l}(k+l)$ must contain in its wake another component of period less than $k+l$. Suppose $m$ is the minimal period $< k+l$ for which there exist components of period $m$ in the wake of $C_{V_l}(k+l)$. Let $C'$ be a component of period $m$ in the wake of $C_{V_l}(k+l)$ closest to $C_{V_l}(k+l)$ along the vein $[C', C_{V_l}(k+l)]$. Then $C'$ must be visible from $C_{V_l}(k)$, so $C'$ must be in the tree of visible components of $C_{V_l}(k)$, contradicting the fact that $C_{V_l}(k+l)$ is a leaf.
 \end{proof}

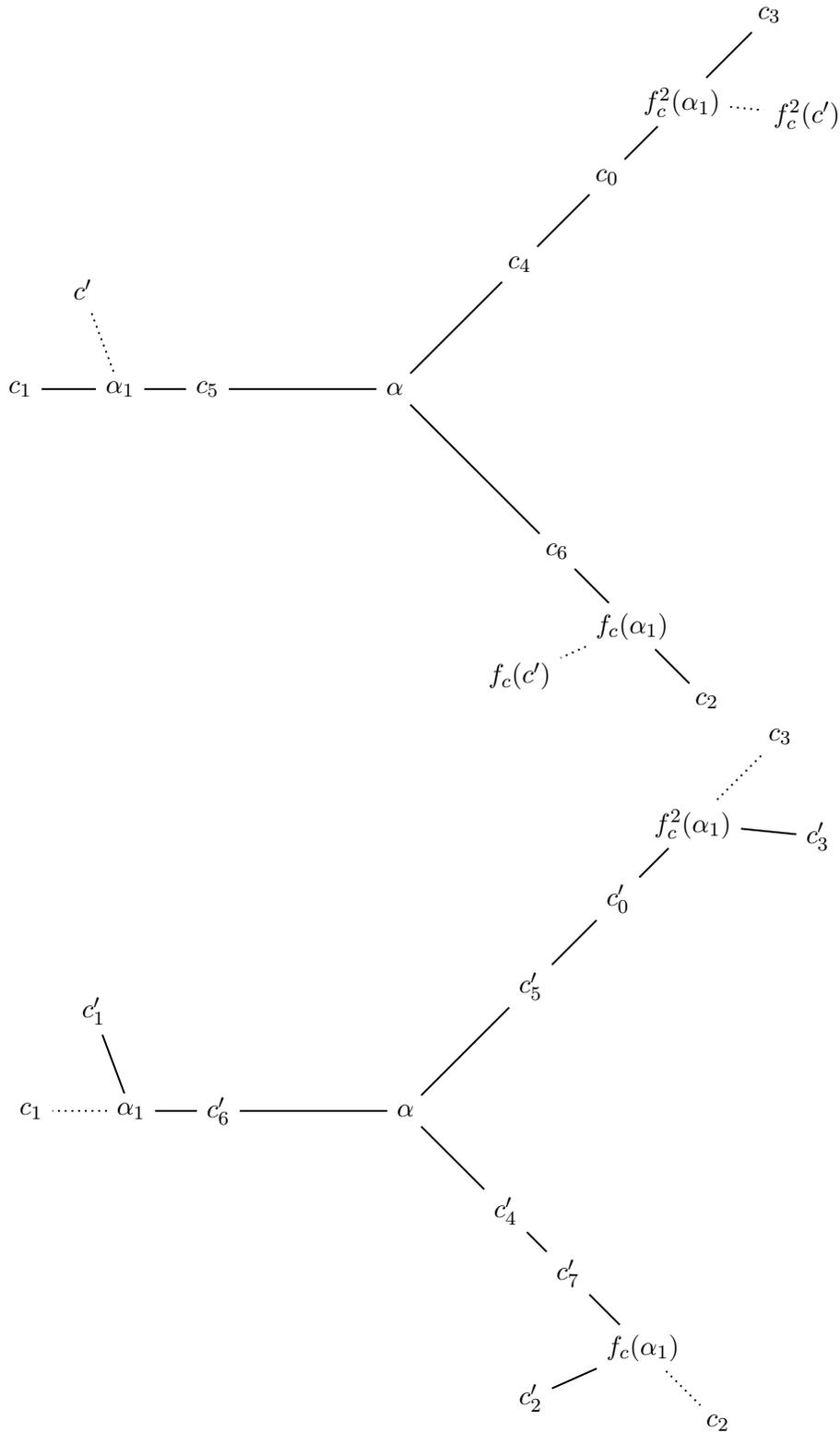
\begin{figure}
\centering
\begin{tikzpicture}[scale=1.8]
\node (0) at (0,0) {$\alpha$}; 
\node (1) at (1,1) {$c_4$};
\node (2) at (-1.5,0) {$c_5$};
\node (3) at (1.3,-1.3) {$c_6$};
\node (4) at (1.7,1.7) {$c_0$};
\node (5) at (-3,0) {$c_1$};
\node (6) at (2.5,-2.5) {$c_2$};
\node (7) at (3,3) {$c_3$};
\node (8) at (-2.2,0) {$\alpha_1$};
\node (9) at (1.9,-1.9) {$f_c(\alpha_1)$};
\node (10) at (2.3,2.3) {$f_c^2(\alpha_1)$};
\node (11) at (-2.5, 0.8) {$c'$};
\node (12) at (1, -2.3) {$f_c(c')$};
\node (13) at (3.3, 2.2) {$f_c^2(c')$};

\foreach \from/\to in {0/1,0/2,0/3,1/4,2/8,5/8,3/9,6/9,4/10,7/10}
    \draw (\from) -- (\to);
\draw[dotted] (8) -- (11);
\draw[dotted] (9) -- (12);
\draw[dotted] (10) -- (13);

\end{tikzpicture} 

\begin{tikzpicture}[scale=1.8]
\node (0) at (0,0) {$\alpha$}; 
\node (1) at (1,1) {$c_5'$};
\node (2) at (-1.5,0) {$c_6'$};
\node (3) at (1.3,-1.3) {$c_7'$};
\node (4) at (1.7,1.7) {$c_0'$};
\node (5) at (-3,0) {$c_1$};
\node (6) at (2.5,-2.5) {$c_2$};
\node (7) at (3,3) {$c_3$};
\node (8) at (-2.2,0) {$\alpha_1$};
\node (9) at (1.9,-1.9) {$f_c(\alpha_1)$};
\node (10) at (2.3,2.3) {$f_c^2(\alpha_1)$};
\node (11) at (-2.5, 0.8) {$c_1'$};
\node (12) at (1, -2.3) {$c_2'$};
\node (13) at (3.3, 2.2) {$c_3'$};
\node (14) at (0.8,-0.8) {$c_4'$};

\foreach \from/\to in {0/1,0/2,0/14,14/3,1/4,2/8,11/8,3/9,12/9,4/10,13/10}
    \draw (\from) -- (\to);
\draw[dotted] (8) -- (5);
\draw[dotted] (9) -- (6);
\draw[dotted] (10) -- (7);

\end{tikzpicture} 
\caption{Schematic diagram of $T_c$ and $T_{c'}$. Top: $T_c$ in solid lines, with dotted lines indicating the edges $J_1, \dots, J_k$. Bottom: $T_c'$ in solid lines, with dotted lines indicating pruned edges of $T_c$.}
\end{figure}

\begin{cor}[Visible Narrow Component]
Suppose $V$ is a $(k,l)$-vein. Then $C_V(k+l)$ is a narrow component whose Hubbard tree is a spiral graph. Moreover the set periods of hyperbolic components of $V$ are exactly the set of $n \geq (k/m-1)(k+l), n \in m\mathbb{N}$, where $m = \gcd(k,l)$. Moreover, for all $n \equiv l \mod k$, $C_V'(n) = C_V(n)$ is minimal.
\end{cor}

\begin{proof}
    This follows from the fact that since $C_V(k+l)$ is a spiral graph, the set of periods of $C_V(k+l)$ is exactly the set $\{n| n \leq_k k+l\}$, which is exactly the set of $\{n | n \geq (k/m-1)(k+l), n \in m\mathbb{N}\}$.
\end{proof}

\begin{prop}[Orbit Forcing]
If $n_1 >_k n_2$ and both $C_V(n_1), C_V(n_2)$ exist such that $n_2\leq_k n+l$ then $C_V(n_1) \succ_V C_V(n_2)$.
\end{prop}

\begin{proof}
First suppose both $1< n_1, n_2 \leq_k k+l$. Then $T_V(n_1)$ and $T_V(n_2)$ are $k$-stars. Label the vertex set of $T_{V}(n_1)$ as $\{p_1, \dots, p_{n_1}\}\cup\{\alpha\}$ with $p_i \mapsto p_{i+1}$ for all $i$ the critical orbit, and $p_1, \dots, p_k$ the vertices closest to $\alpha$ on each of the $k$ segments from $\alpha$ to a leaf of $T_V(n_1)$. Denote the edge $I_i$ as the edge incident with $p_i$ closest to $\alpha$. Then the Markov graph of $T_V(n_1)$ contains the two cycles $I_1 \rightarrow \dots \rightarrow I_{k-1} \rightarrow I_1$ of length $k$, and $I_1 \rightarrow \dots \rightarrow I_{n_1}\rightarrow I_1$ of length $n_1$. So the Markov graph contains cycles of length $ik + j(n_1)$ for all $i \geq 1, j \geq 0$. So $Per(C_V(n_1)) \supseteq \{m | m \leq_k n_2\} \ni n_2$, so $C_V(n_1) \succ_V C_V(n_2)$.

Suppose $C_V(n_1) \not\leq_k k+l$ and $C_V(n_2) \leq_k k+l$. Then $C_V(n_l) \succ_V C_V(k+l)$ and $C_V(k+l) \succeq_V C_V(n_2)$.
\end{proof}

We now prove Theorem \ref{ordering}.

\begin{proof}

    \begin{enumerate}
    
        \item

        First note that $C_V(n)$ cannot exist for $1 < n < k$ because $V$ is on a $p/k$ limb. Additionally, $C_V(n)$ cannot exist for $k < n < n+l$, since $C_V(k+l)$ is a narrow component so $C_V(n) \not \succ_V C_V(k+l)$, and $C_V(k+l)$ is a spiral graph so $C_V(k+l) \not \succ_V C_V(n)$.
        
        First consider $l=1$, and label the critical orbit of $f_{c_l}$ by $\{c_0, c_1, c_2, \dots, c_{k-1}\}$ where $c_0 = 0$, and $c_i \mapsto c_{i+1}$ for $i \leq k-1$, and $c_{k} \mapsto c_{k}$. Then for $c_l$, its Hubbard tree is a $k$-star with vertex set $\{c_0, c_1, \dots, c_k\}\cup \{\alpha\}$, where $c_1, \dots, c_k$ are leaves, $\alpha$ is of degree $k$, $c_0 \in (\alpha, c_k)$. Denote $I_1, \dots, I_k$ the edges incident to $c_1, \dots, c_k$ respectively, $I_0$ the edge $[\alpha, 0]$. The Markov graph consists of the following cycle of length $k+1$ $I_k \mapsto I_0 \mapsto I_1 \mapsto \dots \mapsto I_{k-1} \mapsto I_k$, and the additional arcs $I_{k-1} \mapsto I_0$, $I_k \mapsto I_k$, and $I_k \mapsto I_1$. Then by repeatedly concatenating the loop $I_k \mapsto I_k$ to the cycle of length $k+1$, we can obtain non-elementary cycles of any length greater than $k$. The fact that the only periods less than $k+1$ are $k$ and $1$ is a consequence of the fact that $C_V(k+1)$ is a narrow component visible from $C_V(k)$.
        
        We now suppose $l > 1$. Consider $c_l$ the tip of the $(k,l)$-vein. We can label critical orbit by $\{c_0, c_1, c_2, \dots, c_{k+l-1}\}$ where $c_0 = 0$, and $c_i \mapsto c_{i+1}$ for $i \leq k-1$, and $c_{k+l-1} \mapsto c_{k+l-1}$. Then $c_{k+l-1} \mapsto c_{k+l}$. Since $c_l$ is a tip, each of the points $\{c_1, c_2, \dots, c_{k+l-1}\}$ is a leaf in the Hubbard tree $T_{c_l}$ of $f_{c_l}$. Consider $T_{c_l}$ as an abstract tree, with vertex set given by the critical orbit and all points of degree $>2$ in $T_{c_l}$, as well as $\alpha_{1}$, the other preimage of $\alpha$. The vertex set of $T_{c_l}$ consists of $\{c_0, \dots, c_{k+l-1}\}\cup B$, where $B$ is a finite set consisting of the branch points of $T_{c_l}$ and are preimages of $\alpha$. In particular $\alpha_{-1}$, the preimage of $\alpha$, is in the vertex set, and is located on the sector of $T_{c_l}$ containing $0$.
        
        The edge set of $T_{c_l}$ consists of edges between vertices in $T_{c_l}$. Denote the edges incident to each of the vertices of $c_1, \dots, c_{k+l}$ by $I_1, \dots, I_{k+l}$ respectively, and denote the edge $I_0$ by the edge incident to $c_0$ and $\alpha$. 
        
        I claim that the other endpoint of the edge $I_{k+l}$ is $\alpha_{-1}$. Suppose there is some $\alpha_i \in (0, \beta)$ with degree $>2$ for $T_{c_l}$, such that $\alpha_i$ does not map to $\alpha$. Since $\alpha$ has rotation number $p/k$ the sectors around $\alpha$ are permuted via this rotation. So $\alpha_i \mapsto \alpha_{i+1} \mapsto \dots, \mapsto \alpha_{i+{k-1}}$, with each $\alpha_{i+j}$ in a distinct sector. Consider a local tripod $T$ centered at $\alpha_i$, $T$ is mapped homeomorphically under $f_{c_l}$, so each of $\alpha_i, \alpha_{i+1}, \dots, \alpha_{i+k-1}$ consist of at two children, and thus the sector containing each of these points consists of at least two leaves for a total of at least $2k$ leaves. This contradicts the fact that the critical orbit has preperiod $k+l-1$ for some $l \leq k-1$.
        
        The Markov graph of $T_{c_l}$ consists of the path $I_0 \rightarrow I_1 \rightarrow I_2 \rightarrow I_3 \rightarrow \dots \rightarrow I_{k+l-2}\rightarrow I_{k+l-1}$. Since $\alpha_{-1} \mapsto \alpha$, $I_{k+l} \rightarrow I_0$ and $I_{k+l} \rightarrow I_{k+l}$ are arcs in the Markov graph. Thus $I_0 \rightarrow I_1 \rightarrow I_2 \rightarrow I_3 \rightarrow \dots \rightarrow I_{k+l-1}\rightarrow I_{k+l} \rightarrow I_0$ is a $k+l+1$-cycle in the Markov graph. By concatenating with arbitrary repeats of the arc $I_{k+l}\rightarrow I_{k+l}$, we can obtain cycles of length $k+l+i$ for all $i \geq 1$ in the Markov graph, hence we have periodic orbits of periods of all periods $\geq k+l+1$.

        \item By Proposition \ref{real} and Lemma \ref{principal}, $T_{C_V(n_1)}$ is a $k$-star. If $p_1, \dots, p_{n_1}$ are the critical orbit with $p_1, \dots, p_k$ the vertices of the critical orbit closest to $\alpha$ along each of the $k$ arms of $T_{C_V(n_1)}$, and $p_i \mapsto p_{i+1}$ for all $i$, then vertex set of $T_{C_V(n_1)}$ is $\{p_1, \dots, p_{n_1}\}\cup \{\alpha\}$. If $I_1, \dots, I_{n_1}$ are the edges incident with $p_1, \dots, p_{n_1}$ closest to $\alpha$ where $I_i \rightarrow I_{i+1}$ for all $i$, then the Markov graph contains the $k$-cycle $I_1 \rightarrow \dots \rightarrow I_{k} \rightarrow I_1$ and the $n_1$-cycle $I_1 \rightarrow \dots \rightarrow I_{n_1} \rightarrow I_1$, so there exist periodic orbits of all $m \leq_k n_1$.

        \item The result follows from Lemma \ref{secondary} and the same argument as the previous part.
        
    \end{enumerate}
\end{proof}

We now prove Theorem \ref{dynamics}.

\begin{proof}
    \begin{enumerate}
        \item The result follows from Proposition \ref{real}, Lemma \ref{principal}, and Lemma \ref{secondary}.

        \item If for some other $l' \neq l$ and $i \geq 1$ such that $C_{V_l}(ik+l')$ is also a spiral graph, then its set of periods is $Per(C_{V_l})(ik+l') = \{m | m \leq_k ik+l'\} = \{sk + t(ik+l) | s \geq 1, t \geq 0\}\cup \{1\}$. In particular, $ik+l \not\in Per(C_{V_l}(ik+l'))$, so this means $W_{C_{V_l}(ik+l')} \not\subseteq W_{C_{V_l}(ik+l)}$. But $Per(C_{V_l}(ik+l)) = \{sk + t(ik+l) | s \geq 1, t \geq 0\}\cup \{1\}$, and $ik+l' \not\in Per(C_{V_l}(ik+l))$, so $W_{C_{V_l}(ik+l)}\not\subseteq W_{C_{V_l}(ik+l')}$. Thus the two wakes are disjoint, but this contradicts the fact that any two wakes along the vein must be nested.
        
    \end{enumerate}
\end{proof}

We now prove Theorem \ref{ordering}.

\begin{proof}
\begin{enumerate}

\item This is because for $k = 2l$ elements of $Per(C_V(k+l))$ form a total ordering of $\leq_k$.

\item Since $Per(C_V(ik+l))$ are all spiral graphs, $Per(C_V(jk+l))\setminus Per(C_V((j+1)k+l)) = \{jk+l, (2j+1)k+l, (2j+2)k+l\}$, so $C_V(jk+l) \succ_V \{C_V((2j+1)k+l), C_V((2j+2)k+l)\} \succ_V C_V((2j+2)k+l)$. Since $(2j+1)k+l >_k (2j+2)k+l$ we must have $C_V((2j+1)k+l) \succ_V C_V((2j+2)k+l)$.
\end{enumerate}
\end{proof}

\end{document}